\documentstyle[amssymb,amsfonts]{amsart}

\def\R{{{\mathbf R}}}

\def\eps{\varepsilon}
\newenvironment{proof}{\noindent {\bf Proof} }{\endprf\par}
\def \endprf{\hfill  {\vrule height6pt width6pt depth0pt}\medskip}
\def\emph#1{{\it #1}}
\def\textbf#1{{\bf #1}}

\theoremstyle{plain}
  \newtheorem{theorem}[subsection]{Theorem}

  \newtheorem{proposition}[subsection]{Proposition}
  
  \newtheorem{lemma}[subsection]{Lemma}
  \newtheorem{corollary}[subsection]{Corollary}

\theoremstyle{remark}
  \newtheorem{remark}[subsection]{Remark}

\theoremstyle{definition}

\include{psfig}

\begin{document}

\title{Noncommutative sets of small doubling}

\author{Terence Tao}
\address{Department of Mathematics, UCLA, Los Angeles CA 90095-1555}
\email{tao@@math.ucla.edu}
\thanks{T. Tao is supported by a grant from the MacArthur Foundation, by NSF grant DMS-0649473, and by the NSF Waterman award.  He also thanks the anonymous referee for corrections.}

\begin{abstract}  A corollary of Kneser's theorem, one sees that any finite non-empty subset $A$ of an abelian group $G = (G,+)$ with $|A + A| \leq (2-\eps) |A|$ can be covered by at most $\frac{2}{\eps}-1$ translates of a finite group $H$ of cardinality at most $(2-\eps)|A|$.  Using some arguments of Hamidoune, we establish an analogue in the noncommutative setting.  Namely, if $A$ is a finite non-empty subset of a nonabelian group $G = (G,\cdot)$ such that $|A \cdot A| \leq (2-\eps) |A|$, then $A$ is either contained in a right-coset of a finite group $H$ of cardinality at most $\frac{2}{\eps}|A|$, or can be covered by at most $\frac{2}{\eps}-1$ right-cosets of a finite group $H$ of cardinality at most $|A|$.  We also note some connections with some recent work of Sanders and of Petridis.
\end{abstract}

\maketitle

\section{Introduction} 

A theorem of Kneser \cite{kneser} asserts that if $A, B$ are finite non-empty subsets of an additive group $G$, then the cardinality $|A+B|$ of the sumset $A+B := \{a+b: a \in A, b \in B\}$ of $A$, $B$ obeys the lower bound
\begin{equation}\label{kneser}
 |A+B| \geq |A| + |B| - |H|
\end{equation}
where $H$ is the \emph{symmetry group} $H = \{ h \in G: A+B+h = A+B \}$ of $A+B$.  It leads to the following corollary, describing sets of additive doubling constant strictly less than $2$:

\begin{corollary}\label{kn}  Let $A$ be a finite non-empty subset of an additive group $G = (G,+)$ such that $|A+A| \leq (2-\eps)|A|$ for some $\eps>0$.  Then there exists a finite group $H$ with $|H| \leq (2-\eps) |A|$, such that $A+A$ is covered by at most $\frac{2}{\eps}-1$ translates of $H$.
\end{corollary}

\begin{proof}  Applying \eqref{kneser}, we see that the symmetry group $H$ of $A+A$ has cardinality at least $\eps |A|$.  Since $A+A$ is $H$-invariant, it can be covered by $\frac{|A+A|}{|H|}$ translates of $H$, and the claim follows.
\end{proof}

Informally, this corollary asserts that sets of doubling less than $2$ can be ``controlled'' in some sense by finite groups $H$, with the control deteriorating as the doubling constant approaches $2$.  The quantity $\frac{2}{\eps}-1$ appearing in this corollary is sharp, as can be seen by considering the example of the progression $A = \{1,\ldots,N\}$ with $\eps = \frac{1}{N}$.

In the nonabelian setting, Kneser's theorem is known to fail; see for instance \cite{olson} for some counterexamples.  Nevertheless, by using arguments of Hamidoune \cite{hamidoune}, one can at least partially extend Corollary \ref{kn} to this setting:

\begin{theorem}[Weak Kneser-type theorem]\label{main}  Let $A, S$ be a finite non-empty subset of a multiplicative group $G = (G,\cdot)$ such that $|A| \geq |S|$ and $|A \cdot S| \leq (2-\eps)|S|$ for some $\eps>0$.  Then one of the following statements hold:
\begin{itemize}
\item $S$ is contained in a right-coset of a finite group $H$ with $|H| \leq \frac{2}{\eps} |S|$;
\item $S$ is covered by at most $\frac{2}{\eps}-1$ right cosets of a finite group $H$ with $|H| \leq |S|$.
\end{itemize}
\end{theorem}

This result follows from the methods of Hamidoune \cite{hamidoune}; as it is not explicitly stated there, we give the full proof below for the convenience of the reader.  Similar structural results were obtained in \cite{hamidoune} for sets of doubling constant less than $\frac{8}{3}$, and in \cite{freiman} for sets of doubling constant less than $\frac{1+\sqrt{5}}{2}$.  Sets $S$ of doubling less than $\frac{3}{2}$ can be described completely; see \cite{freiman}.

\section{Hamidoune connectivity}

The proof of Theorem \ref{main} relies on the concept of the \emph{connectivity} of a finite subset $S$ of a group $G$, as developed by Hamidoune
\cite{ham1}-\cite{hls}.  We generalise Hamidoune's theory slightly by introducing an additional real parameter $K$.

More precisely, given a finite subset $S$ of a multiplicative group $G$ and a real number $K$, we define the quantities $c(A) = c_{K,S}(A)$ for finite (and possibly empty) subsets of $G$ by the formula
\begin{equation}\label{cdef}
 c(A) := |A\cdot S| - K|A|.
\end{equation}
This measures the extent to which $S$ ``expands'' $A$, compared against the reference expansion constant $K$.  Clearly, this quantity is left-invariant, thus
\begin{equation}\label{left}
 c( x \cdot A ) = c( A )
\end{equation}
for all finite sets $A$ and $x \in G$.  It also obeys an important \emph{submodularity inequality} (which was also implicitly exploited recently by Petridis \cite{petridis}):

\begin{lemma}[Submodularity]\label{submodular}  For any finite subsets $A, B, S$ of $G$, and any $K \in \R$ one has
$$ c( A \cup B ) + c( A \cap B ) \leq c( A ) + c( B ).$$
\end{lemma}

\begin{proof}  From the inclusion-exclusion principle one has
\begin{equation}\label{aub}
 |A \cup B| + |A \cap B| = |A| + |B|
\end{equation}
and
$$ |(A \cdot S) \cup (B \cdot S)| + |(A \cdot S) \cap (B \cdot S)| = |A \cdot S| + |B \cdot S|.$$
Since
$$ (A \cup B) \cdot S = (A \cdot S) \cup (B \cdot S)$$
and
$$ (A \cap B) \cdot S \subset (A \cdot S) \cap (B \cdot S)$$
we thus have
$$ |(A \cup B) \cdot S| + |(A \cap B) \cdot S| \leq |A \cdot S| + |B \cdot S|,$$
and the claim follows after subtracting $K$ copies of \eqref{aub}.
\end{proof}

Following\footnote{Hamidoune only considered the $K=1$ case, but much of his machinery extends to the case $K \leq 1$, and in fact becomes slightly simpler for $K<1$.} Hamidoune \cite{hamidoune}, we make the following definitions, for fixed $K$ and $S$:
\begin{itemize}
\item The \emph{connectivity} $\kappa = \kappa_K(S)$ is the infimum of $c(A)$ over all finite non-empty sets $A$.  
\item A \emph{fragment} is a finite non-empty set $A$ that attains the infimum $\kappa$: $c(A) = \kappa$.
\item An \emph{atom} is a finite non-empty fragment $A$ of minimal cardinality.
\end{itemize}

From \eqref{left} we see that any left-translate of a fragment is a fragment, and any left-translate of an atom is an atom.

If $K<1$, then we have
\begin{equation}\label{cka-0}
  c(A) \geq (1-K)|A|
\end{equation}
for any $A$.  In particular, $c(A)$ is always positive for non-empty $A$, and takes on a discrete set of values.  This implies that when $K < 1$, then $\kappa$ is positive, hence at least one fragment exists, which (by the well-ordering principle) implies that at least one atom exists.

Let $A$ and $B$ be fragments with non-empty intersection, then from Lemma \ref{submodular} we have
$$ c(A \cup B) + c(A \cap B) \leq c(A) + c(B) = 2\kappa.$$
On the other hand, since $A \cup B$ and $A \cap B$ are finite and non-empty, we have $c(A \cup B), c(A \cap B) \geq \kappa$.  This forces $c(A \cup B) = c(A \cap B) = \kappa$, and so $A \cup B$ and $A \cap B$ are also fragments.  Specialising to atoms (which, by definition, do not contain any strictly smaller fragments), we conclude that any two atoms $A, B$ are either equal or disjoint.  From this and the left-invariance of the atoms, we see that there is a unique atom $H$ that contains the identity.  This atom is either equal or disjoint to any of its left-translates, which implies that $H$ is a finite group.

We summarise the above discussion as follows:

\begin{proposition}  Let $K < 1$.  Then there exists a finite group $H$, such that every left-coset of $H$ is an atom (and furthermore, these are the only atoms).
\end{proposition}

There are more complicated analogues of this proposition for $K=1$; see \cite{hamidoune}.

\section{Proof of Theorem \ref{main}}

We may now prove Theorem \ref{main}.  Let $A, S, \eps, G$ be as in that theorem.  We will take $K$ to be the quantity
$$ K := 1 - \eps/2.$$
From \eqref{cka-0} we then have
\begin{equation}\label{cka}
  c(A) \geq \eps |A| / 2
\end{equation}
for all $A$.

Now we use the hypothesis that $|A| \geq |S|$ and $|A \cdot S| \leq (2-\eps) |S|$, which implies that
$$ c(A) \leq (2-\eps)|S| - (1-\eps/2) |S| = (1-\eps/2) |S|.$$
In particular, we have $c(H) = \kappa \leq c(A) \leq (1 - \eps/2) |S|$, which by \eqref{cka-0} implies an upper bound on $H$:
$$
|H| \leq (\frac{2}{\eps} - 1) |S|.$$
This concludes the proof if $S$ is contained in a single right-coset of $H$, so suppose that $S$ intersects at least two such right-cosets.
We then expand the inequality $c(H) \leq (1-\eps/2) |S|$ as
\begin{equation}\label{hs}
 |H \cdot S| \leq (1 - \eps/2) |S| + (1-\eps/2) |H|.
\end{equation}
As $S$ intersects at least two right-cosets, we have $|H \cdot S| \geq 2|H|$, and so $|H| \leq |S|$.  Also, if we bound $|S|$ in \eqref{hs} crudely by $|H \cdot S|$ and rearrange, we conclude that
$$ |H \cdot S| \leq (\frac{2}{\eps} - 1) |H|.$$
We conclude that $H \cdot S$ (and hence $S$) can be covered by at most $\frac{2}{\eps} - 1$ translates of $H$, and the claim follows.

\begin{remark} It is possible to obtain some further structural control on $S$ under these sorts of hypotheses, by variants of the above method; see \cite{hamidoune}.
\end{remark}

\section{A result of Petridis}

We observe that we can rephrase a recent argument of Petridis \cite[Theorem 1.5]{petridis} using Lemma \ref{submodular}.  Namely, we reprove

\begin{theorem} Let $A, S$ be finite non-empty subsets of a multiplicative group $G$ such that $|A\cdot S| \leq \alpha |A|$ for some $\alpha \in \R$.  Then there exists a non-empty subset $X$ of $A$ such that $|C\cdot X\cdot S| \leq \alpha |C\cdot X|$ for all finite subsets $C$ of $G$.
\end{theorem}

\begin{proof}  Let $X$ be a non-empty subset of $A$ that minimises the quantity
$$ K := \frac{|X \cdot S|}{|X|}$$
over all non-empty subsets of $X$.  Then $K \leq \alpha$, and using the quantities $c()$ defined in \eqref{cdef}, we have
$$ c(X) = 0$$
and
$$ c(Y) \geq 0$$
for all subsets $Y$ of $X$ (including the empty set).  Applying Lemma \ref{submodular}, we see that
$$ c( X \cup Z ) \leq c(Z)$$
for any finite subset $Z$ of $G$; by left-invariance, we more generally have
$$ c( gX \cup Z ) \leq c(Z)$$
for any finite $Z \subset G$ and every $g \in G$.  Iterating this, we see that
$$ c( C\cdot X \cup Z ) \leq c(Z)$$
for all finite sets $C, Z \subset G$; specialising $Z$ to be the empty set, we conclude that
$$ c(C\cdot X) \leq 0$$
and thus
$$ |C \cdot X \cdot S| \leq K |C \cdot X|$$
and the claim follows.
\end{proof}

\section{An argument of Sanders}

In this section we record a Fourier-analytic argument of Sanders (private communication) that obtains a weaker qualitative version of Theorem \ref{main}, but also illustrates a connection between discrete results in additive combinatorics and their continuous counterparts.  To motivate this argument, let us first establish a continuous qualitative analogue of this theorem, in the context of open precompact sets in a locally compact group $G$ with a bi-invariant non-trivial Haar measure $\mu$.  

\begin{proposition} Let $G$ be a locally compact Hausdorff group with a bi-invariant Haar measure $\mu$, and let $A \subset G$ be an open precompact non-empty subset of $G$ such that $\mu( A^{-1} \cdot A ) \leq (2-\eps) \mu( A )$ for some $\eps>0$.  Then there exists a compact open subgroup $H$ of $G$, such that $A \cdot A^{-1}$ is the union of finitely many right cosets of $H$.
\end{proposition}

\begin{proof}  We consider the convolution
\begin{align*}
 f(x) &:= \frac{1}{\mu(A)} 1_{A} * 1_{A^{-1}}(x) \\
&= \frac{1}{\mu(A)} \int_G 1_{A}(y) 1_{A^{-1}}( y^{-1} x )\ d\mu(y) \\
&= \frac{1}{\mu(A)} \mu( A \cap x \cdot A ).
\end{align*}
Since $1_A, 1_{A^{-1}}$ are bounded, compactly supported functions, we see that the convolution $f = 1_A * 1_{A^{-1}}$ is a continuous, compactly supported function.  If $x$ lies in the support $A \cdot A^{-1}$ of $f$, then $x \in A \cdot A^{-1}$, and thus $x = ab^{-1}$ for some $a,b \in A$.  But then
$$ 
f(x) = \frac{1}{\mu(A)} \mu( (a^{-1} \cdot A) \cap (b^{-1} \cdot A) ).
$$ 
Since $a \cdot A$ and $b \cdot A$ both lie in $A^{-1} \cdot A$ and have measure $\mu(A)$, we see from the hypothesis $\mu(A^{-1} \cdot A) \leq (2-\eps) \mu(A)$ and the inclusion-exclusion principle that we thus have the uniform lower bound
\begin{equation}\label{feps}
f(x) \geq \eps
\end{equation}
on the support $A \cdot A^{-1}$ of $f$.  In other words, there is a ``gap'' in the range of $f$, in that it cannot take on values in the interval $(0,\eps)$.  This gap disconnects the support $A \cdot A^{-1}$ from its complement; both sets become both open and closed.  In particular, $A \cdot A^{-1}$ is now compact.  By continuity and compactness, this implies that there exists a neighbourhood $U$ of the identity such that $U \cdot A \cdot A^{-1} = A \cdot A^{-1}$.  Letting $H$ be the group generated by $U$, we conclude that $H$ is open and contained in the compact set $A \cdot A^{-1}$, and thus must also be compact, with $A \cdot A^{-1}$ the union of finitely many right-cosets of $H$ as required.
\end{proof}

Now we return to the discrete setting, and establish the following weaker version of Theorem \ref{main}:

\begin{theorem}\label{sand} Let $A$ be a finite subset of a multiplicative group $G$ such that $|A^{-1} \cdot A| \leq (2-\eps) |A|$.  Then there exists a finite subgroup $H$ with $|H| \ll_\eps |A|$ such that $A$ can be covered by $O_\eps(1)$ right-cosets of $H$.
\end{theorem}

Here we use the asymptotic notation $X \ll Y$, $Y \gg X$, or $X=O(Y)$ to denote the assertion $X \leq CY$ for some absolute constant $C$; if we need the implied constant $C$ to depend on a parameter, we indicate this by subscripts, e.g. $X \ll_\eps Y$ denotes assertion $X \leq C_\eps Y$ for some constant $C_\eps$ depending on $\eps$.

To prove this theorem, we analyse the function
\begin{align*}
 f(x) &:= \frac{1}{|A|} 1_A * 1_{A^{-1}}(x)\\
&:= \frac{1}{|A|} \sum_{y \in G} 1_A(y) 1_A( y^{-1} x )\\
&= \frac{1}{|A|} |A \cap x \cdot A|.
\end{align*}

As in the continuous case, we can show that $f$ is bounded away from zero on its support, in the sense that
\begin{equation}\label{fip}
 f(x) \geq \eps
\end{equation}
for all $x$ in the support $A \cdot A^{-1}$ of $f$. So we once again have a gap in the range of $f$. However, in this discrete setting, we do not have any obvious quantitative control on the ``continuity'' of the convolution $f$ to exploit this gap (this is ultimately because $A$ does not have good ``measurability'' properties).  However, it turns out that $f$ is controlled in a certain \emph{Wiener algebra} $A(G)$, which roughly speaking is the non-commutative analogue of functions with absolutely convergent Fourier transform.  In the abelian setting, the fact that we have control in the Wiener algebra is a consequence of Plancherel's theorem (that asserts that $L^2$ functions have square-summable Fourier coefficients), the Cauchy-Schwarz inequality, and the observation that convolution corresponds to pointwise multiplication of Fourier coefficients.  It turns out that an analogous statement can be made in the non-abelian case.  

In the classical setting of Fourier series, functions in the Wiener algebra have absolutely convergent Fourier series, and in particular are necessarily continuous.  A deep result of Sanders \cite{sanders} asserts, roughly speaking, that in more general non-abelian groups $G$, functions in the Wiener algebra $A(G)$ can be uniformly approximated by continuous functions ``outside of a set of negligible measure''.  A precise version of this statement is as follows:

\begin{proposition}[Almost uniform approximation by continuous functions]  Let $f$ be as above, and let $\sigma > 0$.  Then there exists a symmetric neighbourhood $S'$ of the identity with $|S'| \gg_\sigma |A|$ and a function $F: G \to \R^+$ such that
\begin{equation}\label{das}
 |F(s' x) - F(x)| \ll \sigma
\end{equation}
and
\begin{equation}\label{fss}
(\sum_{s' \in S'} |F(s' x) - f(s'x)|^2)^{1/2} \ll \sigma |S'|^{1/2}
\end{equation}
for all $s' \in S'$ and $x \in G$.
\end{proposition}

\begin{proof} See \cite[Proposition 20.1]{sanders}.  The ``continuous'' function $F$ is in fact of the form $F := \frac{1}{|S|} 1_S * \frac{1}{|S|} 1_S * f$ for some set $S$ larger than $S'$ (but small compared to $A$).
\end{proof}

\begin{remark} Sanders' paper uses a nontrivial amount of spectral theory (or nonabelian Fourier analysis).  It is possible to use ``softer'' (and significantly simpler) methods to obtain weaker regularity results on convolutions (giving ``$L^p$ continuity'' rather than ``$L^\infty$ continuity''), as was done in \cite{croot}, but unfortunately it does not appear that these easier results suffice for the application here (which relies on iterating the control given by continuity, and so cannot handle the small exceptional sets allowed by an $L^p$ regularity result).
\end{remark}

We return to the proof of Theorem \ref{sand}.
Let $\sigma > 0$ be a sufficiently small parameter depending on $\eps$ to be chosen later, and let $F$ be the approximation to $f$ given above.  From \eqref{das}, \eqref{fss} we have
$$ (\sum_{s' \in S'} |F(x) - f(s'x)|^2)^{1/2} \ll \sigma |S'|^{1/2}.$$
From this and the gap property \eqref{fip} we see that $F$ also has a gap, in that it cannot take values in the interval $[\eps/3,2\eps/3]$ (say) if $\sigma$ is small enough.  In particular, from this and \eqref{das} we see that the set $\Omega := \{ x \in G: F(x) > 2\eps/3 \}$ is closed under left-multiplication by $S'$, if $\sigma$ is small enough; since this set contains $1$, it must therefore contain the group $H$ generated by $S'$.  On the other hand, $F$ has an $\ell^1$ norm of $O(|A|)$, and so by Markov's inequality we have 
$$|H| \leq |\Omega| \ll |A| / \eps;$$
in the other direction we have
$$ |H| \geq |S'| \gg_\sigma |A|,$$
so $H$ is of comparable size to $A$.    As $F$ is large ($\gg \eps$) on $H$, we can use \eqref{fss} to ensure (for $\sigma$ small enough) that $f$ is also large on $H$ on average (specifically, $\|f\|_{\ell^2(H)} \gg |H|^{1/2} / \eps$); this ensures from the pigeonhole principle that some right-translate of $A$ has large intersection (cardinality $\gg |H|/\eps$) with $H$, and this combined with the bounded doubling of $A$ ensures that $A$ is covered by $O( \frac{|A|}{\eps |H|} ) = O_{\eps,\sigma}(1)$ right-translates of $H$, and the claim follows.

\end{document}